\documentclass[]{theclass}
\usepackage{mathtools}
\definecolor{orchid}{rgb}{0.85, 0.44, 0.84}

\begin{document}
\begin{frontmatter}

\titledata{Perfect matchings, Hamiltonian cycles and edge-colourings in a class of cubic graphs}{}           

\authordata{Mari\'{e}n Abreu}
{Dipartimento di Matematica, Informatica ed Economia\\ Universit\`{a} degli Studi della Basilicata, Italy}
{marien.abreu@unibas.it}{}

\authordata{John Baptist Gauci}
{Department of Mathematics, University of Malta, Malta}
{john-baptist.gauci@um.edu.mt}
{}

\authordatatwo{Domenico Labbate}{domenico.labbate@unibas.it}{}{Federico Romaniello}{federico.romaniello@unibas.it}{}
{Dipartimento di Matematica, Informatica ed Economia\\ Universit\`{a} degli Studi della Basilicata, Italy}

\authordata{Jean Paul Zerafa}
{Department of Technology and Entrepreneurship Education\\University of Malta, Malta;\\ 
Department of Computer Science, Faculty of Mathematics, Physics and Informatics\\ Comenius University, Mlynsk\'{a} Dolina, 842 48 Bratislava, Slovakia}
{zerafa.jp@gmail.com}
{The author was partially supported by VEGA 1/0813/18 and VEGA 1/0743/21.}

\keywords{Cubic graph, perfect matching, Hamiltonian cycle, 3-edge-colouring}
\msc{05C15, 05C45, 05C70}

\begin{abstract}
A graph $G$ has the Perfect-Matching-Hamiltonian property (PMH-property) if for each one of its perfect matchings, there is another perfect matching of $G$ such that the union of the two perfect matchings yields a Hamiltonian cycle of $G$. The study of graphs that have the PMH-property, initiated in the 1970s by Las Vergnas and H\"{a}ggkvist, combines three well-studied properties of graphs, namely matchings, Hamiltonicity and edge-colourings. In this work, we study these concepts for cubic graphs in an attempt to characterise those cubic graphs for which every perfect matching corresponds to one of the colours of a proper 3-edge-colouring of the graph. We discuss that this is equivalent to saying that such graphs are even-2-factorable (E2F), that is, all 2-factors of the graph contain only even cycles. The case for bipartite cubic graphs is trivial, since if $G$ is bipartite then it is E2F. Thus, we restrict our attention to non-bipartite cubic graphs. A sufficient, but not necessary, condition for a cubic graph to be E2F is that it has the PMH-property. The aim of this work is to introduce an infinite family of E2F non-bipartite cubic graphs on two parameters, which we coin \emph{papillon graphs}, and determine the values of the respective parameters for which these graphs have the PMH-property or are just E2F. We also show that no two papillon graphs with different parameters are isomorphic.
\end{abstract}

\end{frontmatter}

\section{Introduction}

Let $G$ be a simple connected graph of even order with vertex set $V(G)$ and edge set $E(G)$. A \emph{$k$-factor} of $G$ is a $k$-regular spanning subgraph of $G$ (not necessarily connected). Two very well-studied concepts in graph theory are \emph{perfect matchings} and \emph{Hamiltonian cycles}, where the former is the edge set of a $1$-factor and the latter is a connected $2$-factor of a graph. For $t\geq 3$, a \emph{cycle} of length $t$ (or a $t$-cycle), denoted by $C_{t}=(v_{1}, \ldots, v_{t})$, is a sequence of mutually distinct vertices $v_1,v_2,\ldots,v_{t}$ with corresponding edge set $\{v_{1}v_{2}, \ldots, v_{t-1}v_{t}, v_{t}v_{1}\}$. For definitions not explicitly stated here we refer the reader to \cite{Diestel}.
A graph $G$ admitting a perfect matching is said to have the \emph{Perfect-Matching-Hamiltonian property} (for short the \emph{PMH-property}) if for every perfect matching $M$ of $G$ there exists another perfect matching $N$ of $G$ such that the edges of $M\cup N$ induce a Hamiltonian cycle of $G$. For simplicity, a graph admitting this property is said to be PMH. This property was first studied in the 1970s by Las Vergnas \cite{LasVergnas} and H\"aggkvist \cite{Haggkvist}, and for more recent results about the PMH-property we suggest the reader to \cite{pmhlinegraphs,rook,ThomassenEtAl,cpcq,accordions}.
In \cite{ThomassenEtAl}, a property stronger than the PMH-property is studied: the Pairing-Hamiltonian property, for short the PH-property. Before proceeding to the definition of this property, we first define what a pairing is. For any graph $G$, $K_{G}$ denotes the complete graph on the same vertex set $V(G)$ of $G$. A perfect matching of $K_{G}$ is said to be a \emph{pairing} of $G$, and a graph $G$ is said to have the \emph{\mbox{Pairing-Hamiltonian} property} if every pairing $M$ of $G$ can be extended to a Hamiltonian cycle $H$ of $K_{G}$ such that $E(H)-M\subseteq E(G)$. Clearly, a graph having the PH-property is also PMH, although the converse is not necessarily true. Amongst other results, the authors of \cite{ThomassenEtAl} show that the only cubic graphs admitting the PH-property are the complete graph $K_{4}$, the complete bipartite graph $K_{3,3}$, and the cube $\mathcal{Q}_{3}$. However, this does not mean that these are the only three cubic graphs admitting the PMH-property. For instance, all cubic 2-factor Hamiltonian graphs (all 2-factors of such a graph form a Hamiltonian cycle) are PMH (see for example \cite{StarProduct, m1freg, charm1freg, m1fcub, m1fdetper}). 

If a cubic graph $G$ is PMH, then every perfect matching of $G$ corresponds to one of the colours of a (proper) $3$-edge-colouring of the graph, and we say that every perfect matching can be extended to a $3$-edge-colouring. This is achieved by alternately colouring the edges of the Hamiltonian cycle containing a predetermined perfect matching using two colours, and then colouring the edges not belonging to the Hamiltonian cycle using a third colour.
However, there are cubic graphs which are not PMH but have every one of their perfect matchings that can be extended to a $3$-edge-colouring (see for example Figure \ref{FigureC6K2}). The following proposition characterises all cubic graphs for which every one of their perfect matchings can be extended to a 3-edge-colouring of the graph.

\begin{figure}[h]
\centering
\includegraphics[width=0.3\textwidth]{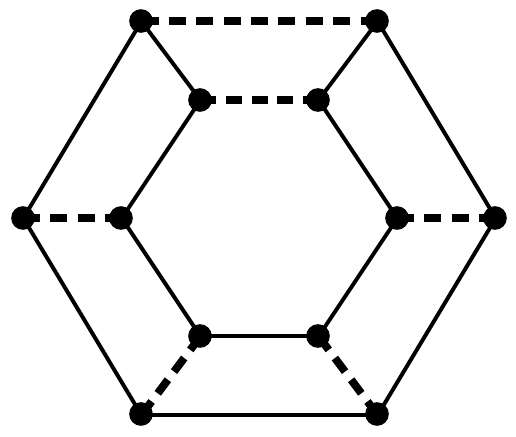}
\caption{The bold dashed edges can be extended to a proper 3-edge-colouring but not to a Hamiltonian cycle}
\label{FigureC6K2}
\end{figure}

\begin{proposition}\label{prop e2f}
Let $G$ be a cubic graph admitting a perfect matching. Every perfect matching of $G$ can be extended to a $3$-edge-colouring of $G$ if and only if all $2$-factors of $G$ contain only even cycles.
\end{proposition}

\begin{proof}
Let $F$ be a $2$-factor of $G$, and let $M$ be the perfect matching $E(G)-E(F)$. Since $M$ can be extended to a 3-edge-colouring of $G$, $F$ can be 2-edge-coloured, and hence $F$ does not contain any odd cycles. Conversely, let $M'$ be a perfect matching of $G$, and let $F'$ be its complementary 2-factor, that is, $E(F')=E(G)-M'$. Since $F'$ contains only even cycles, $M'$ can be extended to a $3$-edge-colouring, by assigning a first colour to all of its edges and then alternately colouring the edges of the 2-factor $F'$ using another two colours.
\end{proof}

We shall call graphs in which all $2$-factors consist only of even cycles as \emph{even-$2$-factorable} graphs, denoted by E2F for short. In particular, from Proposition \ref{prop e2f}, if a cubic graph $G$ has the PMH-property, then it is also E2F. As in the proof of Proposition \ref{prop e2f}, in the sequel, given a perfect matching $M$ of a cubic graph $G$, the 2-factor obtained after deleting the edges of $M$ from $G$ is referred to as the \emph{complementary 2-factor} of $M$.

If a cubic graph is bipartite, then trivially, each of its perfect matchings can be extended to a $3$-edge-colouring, since it is E2F. But what about non-bipartite cubic graphs? In Table \ref{table jan}, we give the number of non-isomorphic non-bipartite 3-connected cubic graphs (having girth at least $4$) such that each one of their perfect matchings can be extended to a $3$-edge-colouring. As is the case of \emph{snarks} (bridgeless cubic graphs which are not $3$-edge-colourable), these seem to be difficult to find, as one can notice after comparing these numbers to the total number of non-isomorphic $3$-edge-colourable (Class I) non-bipartite $3$-connected cubic graphs having girth at least $4$, also given in Table \ref{table jan}. The numbers shown in this table were obtained thanks to a computer check done by Jan Goedgebeur, and the data is sorted according to the cyclic connectivity of the graphs considered. We remark that E2F cubic graphs having girth $3$ can be obtained by applying a star product between an E2F
cubic graph of smaller order and the complete graph $K_{4}$---this has been investigated further by the last two authors in \cite{betwixt}. This is the reason why only graphs having girth at least 4 are considered in this work. More results on star products (also known in the literature as $3$-cut connections) in cubic graphs can be found in \cite{StarProduct, m1freg, HS, charm1freg, m1fcub, m1fdetper}. 

\begin{table}[h]
\centering
\begin{tabular}{ccccccccccc}
\addlinespace[-\aboverulesep]
\cmidrule[\heavyrulewidth]{2-10}
      &       & \multicolumn{4}{c}{Cyclic connectivity} & &  \multicolumn{2}{c}{Total no. of graphs}\\
\cmidrule{3-6}\cmidrule{8-9}
 & & $3$             & $4$   & $5$  & $6$ & & E2F & Class I & \shortstack{ratio \\ {\tiny E2F : Class I}} &\\
\cmidrule[\heavyrulewidth]{2-10}
\multirow{8}{*}{\rotatebox[origin=c]{90}{Number of vertices}}
&$8$          &   /    & $1$   & /    & /   & & $1$& $1$& $100\%$&
\\
&$10$         &   /    &  /    & /    & /   & & $0$& $3$& $0\%$&\\
&$12$         & $2$    & $5$   & $2$  & /   & & $9$& $17$& $52.94\%$&\\
&$14$         & $2$    & $2$   & $2$  & /   & & $6$& $92$& $6.52\%$&\\
&$16$         & $35$   & $56$  & $4$  & /   & & $95$& $716$& $13.27\%$&\\
&$18$         & $84$   & $21$  & $9$  & /   & & $114$& $7343$& $1.55\%$&\\
&$20$         & $926$  & $655$ & $15$ & $2$ & & $1598$& $93946$& $1.70\%$&\\
&$22$         & $2978$ & $331$ & $17$ & $6$ & & $3332$& $1400203$& $0.24\%$&\\
\cmidrule[\heavyrulewidth]{2-10}
\end{tabular}
\caption{The number of non-isomorphic non-bipartite $3$-connected cubic graphs with girth at least $4$ which are E2F and Class I}
\label{table jan}
\end{table}

A complete characterisation of which cubic graphs are PMH is still elusive, so considering the Class I non-bipartite cubic graphs having the property that each one of their perfect matchings can be extended to a $3$-edge-colouring may look presumptuous. As far as we know this property and the corresponding characterisation problem were never considered before and tackling the following problem seems a reasonable step to take.

\begin{problem}\label{problem pm3ec}
Characterise the Class I non-bipartite cubic graphs for which each one of their perfect matchings can be extended to a $3$-edge-colouring, that is, are E2F.
\end{problem}

We remark that although the PMH-property is an appealing property in its own right, Problem \ref{problem pm3ec} continues to justify its study in relation to cubic graphs. Observe that in the family of cubic graphs, whilst snarks are not $3$-edge-colourable, even-$2$-factorable graphs are quite the opposite being ``very much 3-edge-colourable", since the latter can be 3-edge-coloured by assigning a colour to one of its perfect matchings, and then alternately colour the edges of the complementary 2-factor.

\subsection{Cycle permutation graphs}

Consider two disjoint cycles each of length $t$, referred to as the first and second $t$-cycles and denoted by $(x_{1},\ldots, x_{t})$ and $(y_{1},\ldots, y_{t})$, respectively. Let $\sigma$ be a permutation of the symmetric group $\mathcal{S}_{t}$ on the $t$ symbols $\{1,\ldots, t\}$. The \emph{cycle permutation graph corresponding to $\sigma$} is the cubic graph obtained by considering the first and second $t$-cycles in which $x_{i}$ is adjacent to $y_{\sigma(i)}$, where $\sigma(i)$ is the image of $i$ under the permutation $\sigma$.

\begin{figure}[h]
\centering
\includegraphics[width=0.6\textwidth]{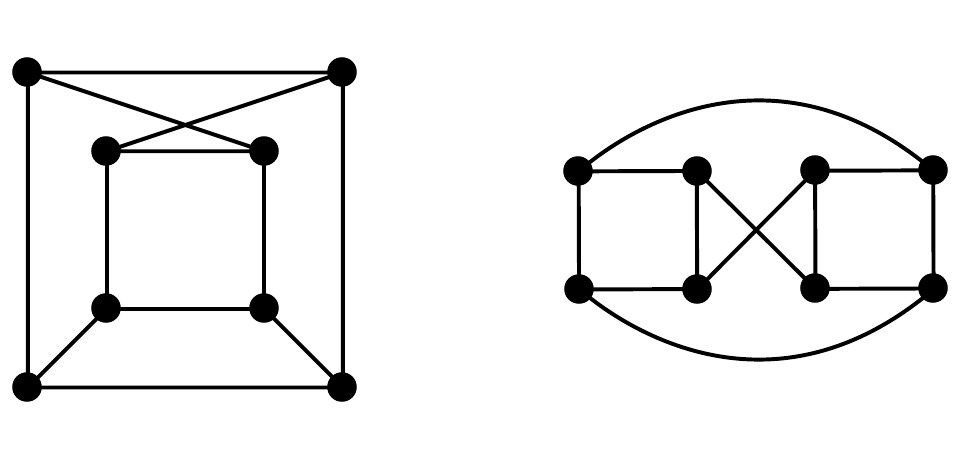}
\caption{Two different drawings of the smallest non-bipartite E2F cubic graph}
\label{figure mobius}
\end{figure}

The smallest non-bipartite cubic graph (from Table \ref{table jan}) which is E2F is in fact a cycle permutation graph corresponding to $\sigma=(1\,\,\,2)\in\mathcal{S}_{4}$, where $\sigma(1)=2, \sigma(2)=1, \sigma(3)=3$, and $\sigma(4)=4$ (see Figure \ref{figure mobius}). This shows that the edges between the vertices of the first and second $4$-cycles of the cycle permutation graph are $x_{1}y_{2}, x_{2}y_{1}, x_{3}y_{3}, x_{4}y_{4}$. In what follows we shall denote permutations in cycle notation and, for simplicity, fixed points shall be suppressed. With the help of Wolfram Mathematica, in Table \ref{table e2f vs pmh} we provide the number of non-isomorphic non-bipartite cycle permutation graphs up to 20 vertices which are PMH or just E2F. Recall that PMH cubic graphs are also E2F, and so, PMH cycle permutation graphs should be searched for from amongst the cycle permutation graphs which are E2F. We also remark that, in the sequel, cycle permutation graphs with total number of vertices equal to twice an odd number are not considered because, in this case, the first and second cycles form a 2-factor consisting of two odd cycles, and so they are trivially not E2F.

\begin{table}[h]
\centering
\begin{tabular}{cccc}
\addlinespace[-\aboverulesep]
\cmidrule[\heavyrulewidth]{2-4}
 & & E2F             & PMH  \\
\cmidrule[\heavyrulewidth]{2-4}
\multirow{4}{*}{\rotatebox[origin=c]{90}{No. of vertices}}
&$8$          & $1$    & $0$  \\
&$12$         & $5$    & $1$  \\
&$16$         & $28$    & $2$  \\
&$20$         & $175$    & $0$  \\
\cmidrule[\heavyrulewidth]{2-4}
\end{tabular}
\caption{The number of non-isomorphic non-bipartite cycle permutation graphs with girth at least 4 which are E2F and PMH}
\label{table e2f vs pmh}
\end{table}

This work is a first structured attempt at tackling Problem \ref{problem pm3ec}. We give an infinite family of non-bipartite cycle permutation graphs which admit the PMH-property or are just E2F. In Section \ref{section papillon graphs}, we generalise the smallest cubic graph which is E2F into a family of non-bipartite cycle permutation graphs, namely papillon graphs $\mathcal{P}_{r,\ell}$ (for $r,\ell\in\mathbb{N}$),  whose smallest member $\mathcal{P}_{1,1}$ is, in fact, the graph in Figure \ref{figure mobius}. We show that papillon graphs are E2F for all values of $r$ and $\ell$ (Theorem \ref{theorem papillon e2f}) and PMH if and only if both $r$ and $\ell$ are even (Theorem \ref{theorem papillon pmh} and Theorem \ref{pmh unbalanced}).

\section{Papillon graphs}\label{section papillon graphs}
Let $[n]=\{1,\ldots,n\}$, for some positive integer $n$.

\begin{definition}
Let $r$ and $\ell$ be two positive integers. The \emph{papillon graph} $\mathcal{P}_{r,\ell}$ is the graph on $4r+4\ell$ vertices such that $V(\mathcal{P}_{r,\ell})=\{u_{i},v_{i}:i\in[2r+2\ell]\}$, where:

\begin{enumerate}[(i)]
\item $(u_{1}, u_{2},\ldots, u_{2r+2\ell})$ is a cycle of length $2r+2\ell$;
\item $u_{i}$ is adjacent to $v_{i}$, for each $i\in[2r+2\ell]$; and
\item the adjacencies between the vertices $v_{i}$, for $i\in[2r+2\ell]$, form a cycle of length $2r+2\ell$ given by the edge set \begin{align*}
\{v_{2i-1}v_{2i} : i\in [r+\ell]\}&\cup
\{v_{2i-1}v_{2i+2} : i\in [r+\ell-1]\setminus\{s\}\}\\
&\cup \{v_{2}v_{2s+2}, v_{2s-1}v_{2r+2\ell-1}\},
\end{align*}
where $s=\min\{r,\ell\}$.
\end{enumerate}

\noindent Clearly, the two papillon graphs $\mathcal{P}_{r,\ell}$ and $\mathcal{P}_{\ell,r}$ are isomorphic, and henceforth, without loss of generality, we shall tacitly assume that $r\leq \ell$. The papillon graph $\mathcal{P}_{r,\ell}$ for $r\geq 2$ is depicted in Figure \ref{figure papillon definition}. When $r$ and $\ell$ are equal, say $r=\ell=n$, the papillon graph $\mathcal{P}_{r,\ell}$ is said to be \emph{balanced}, and simply denoted by $\mathcal{P}_{n}$ (see, for example, Figure \ref{figure h3}). Otherwise, $\mathcal{P}_{r,\ell}$ is said to be \emph{unbalanced} (see, for example, Figure \ref{figure unbalanced}). The $(2r+2\ell)$-cycle induced by the sets of vertices $\{u_{i}:i\in[2r+2\ell]\}$ is referred to as the \emph{outer-cycle}, whilst the $(2r+2\ell)$-cycle induced by the vertices $\{v_{i}:i\in[2r+2\ell]\}$ is referred to as the \emph{inner-cycle}. The edges on these two $(2r+2\ell)$-cycles are said to be the \emph{outer-edges} and \emph{inner-edges} accordingly, whilst the edges $u_{i}v_{i}$ are referred to as \emph{spokes}. The edges $u_{1}u_{2r+2\ell}$, $v_{2r-1}v_{2r+2\ell-1}$, $v_{2}v_{2r+2}$, $u_{2r}u_{2r+1}$ are denoted by $a,b,c,d$, respectively, and we shall also denote the set $\{a,b,c,d\}$ by $\mathcal{X}$. The set $\mathcal{X}$ is referred to as the \emph{principal 4-edge-cut} of $\mathcal{P}_{r,\ell}$.
\end{definition}

\begin{figure}[h]
\centering
\includegraphics[width=.75\textwidth]{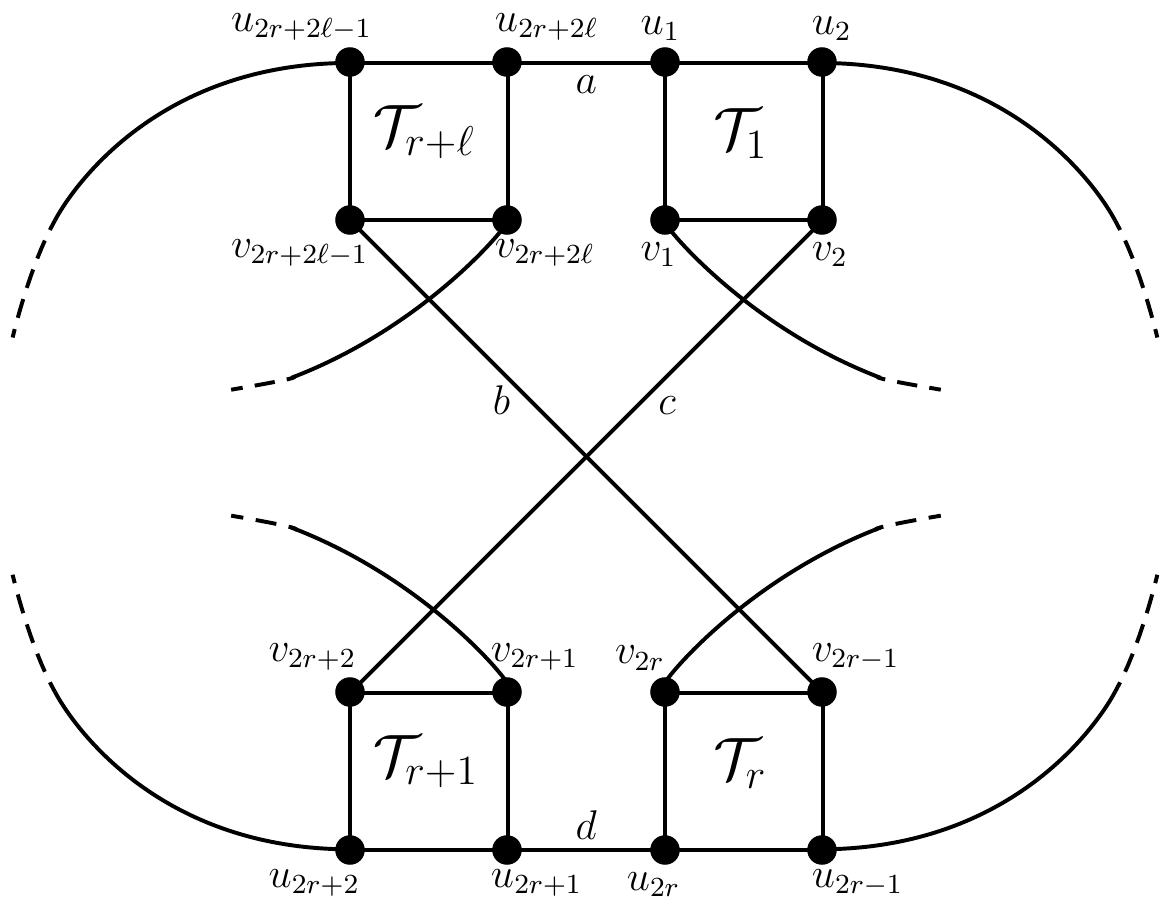}
\caption{The papillon graph $\mathcal{P}_{r,\ell}$, for $\ell\geq r\geq 2$, and the $4$-pole $\mathcal{T}_{j}$, for $j\in [r+\ell]$}
\label{figure papillon definition}
\end{figure}

\noindent The graph in Figure \ref{figure mobius} is actually the smallest (balanced) papillon graph $\mathcal{P}_{1}$. 
In general, since $\{u_{i}:i\in[2r+2\ell]\}$ and $\{v_{i}:i\in[2r+2\ell]\}$ induce two disjoint $(2r+2\ell)$-cycles in $\mathcal{P}_{r,\ell}$, and since every vertex belonging to the outer-cycle is adjacent to exactly one vertex on the inner-cycle,  there exists an isomorphism $\pi$ between the papillon graph $\mathcal{P}_{r,\ell}$ and a cycle permutation graph corresponding to some $\sigma\in\mathcal{S}_{2r+2\ell}$ satisfying $\pi(x_i)=u_i$ and  $\pi(y_i)=v_{\sigma^{-1}(i)}$, for each $i \in [2r+2\ell]$. In fact, the papillon graph $\mathcal{P}_{r,\ell}$ is the cycle permutation graph, with $(u_{1}, \ldots, u_{2r+2\ell})$ as the first cycle, corresponding to the permutation:
\begin{itemize}
\item $\sigma_{1,\ell}:=(3\,\,\,4)\ldots(2\ell+1\,\,\,2\ell+2)$, with fixed points $1$ and $2$, when $\ell\geq 1$; 
\item $\sigma_{2,2}:=(1\,\,\,2)(3\,\,\,4)(5\,\,\,7)(6\,\,\,8)$;
\item $\sigma_{r,3}:=(1\,\,\,2)\ldots(2r-1\,\,\,2r)(2r+1\,\,\,2r+5)(2r+2\,\,\,2r+6)$, with fixed points $2r+3$ and $2r+4$, when $r\in\{2,3\}$; and
\item $\sigma_{r,\ell}:=(1\,\,\,2)\ldots(2r-1\,\,\,2r)(2r+1\,\,\,2r+2\ell-1)(2r+2\,\,\,2r+2\ell)(2r+3\,\,\,2r+2\ell-3)(2r+4\,\,\,2r+2\ell-2)\ldots(\alpha\,\,\,\beta)$, when $\ell\geq r\geq 4$, where $(\alpha\,\,\,\beta)=(2r+\ell\,\,\,2r+\ell+2)$ if $\ell$ is even, and $(\alpha\,\,\,\beta)=(2r+\ell-1\,\,\,2r+\ell+3)$ if $\ell$ is odd.
\end{itemize} 

We remark that when $r>1$, the above permutations has no fixed points when $\ell$ is even, but, when $\ell$ is odd, $2r+\ell$ and $2r+\ell+1$ are fixed points, and thus, in this case, $x_{2r+\ell}$ is adjacent to $y_{2r+\ell}$, and $x_{2r+\ell+1}$ is adjacent to $y_{2r+\ell+1}$ in $\mathcal{P}_{r,\ell}$. Note that since $\sigma_{r,\ell}$ is an involution for all positive integers $r$ and $\ell$, the isomorphism $\pi$ mentioned above can be rewritten as follows: $\pi(x_i)=u_i$ and  $\pi(y_i)=v_{\sigma(i)}$, for each $i \in [2r+2\ell]$. The papillon graph $\mathcal{P}_{r,\ell}$ admits a natural automorphism $\psi$ which exchanges the two cycles, given by $\psi(u_i)=v_{\sigma_{r,\ell}(i)}$ and $\psi(v_i)=u_{\sigma_{r,\ell}(i)}$, for each $i\in[2r+2\ell]$. In fact, the function $\psi$ is clearly bijective. Moreover, it maps edges of the outer-cycle to edges of the inner-cycle (and vice-versa), and maps spokes to spokes, since the edges $u_{i}v_{i}$ are mapped to $u_{\sigma_{r,\ell}(i)}v_{\sigma_{r,\ell}(i)}$.

\begin{figure}[h]
\centering
\includegraphics[width=.75\textwidth]{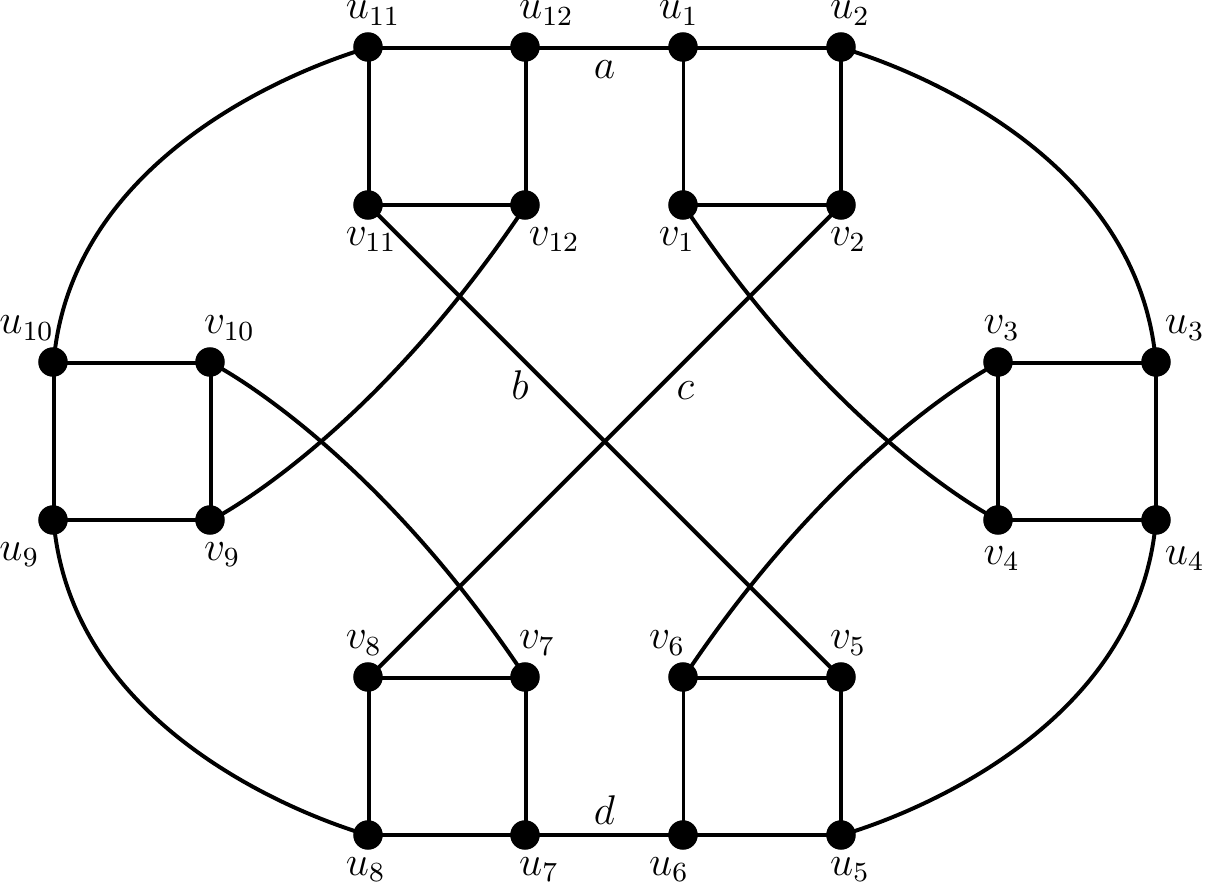}
\caption{The balanced papillon graph $\mathcal{P}_{3}$ on 24 vertices}
\label{figure h3}
\end{figure}

Before proceeding, we introduce multipoles which generalise the notion of graphs. This will become useful when describing papillon graphs. A \emph{multipole} $\mathcal{Z}$ consists of a set of vertices $V(\mathcal{Z})$ and a set of generalised edges such that each generalised edge is either an edge in the usual sense (that is, it has two endvertices) or a semiedge.  A \emph{semiedge} is a generalised edge having exactly one endvertex. The set of semiedges of $\mathcal{Z}$ is denoted by $\partial\mathcal{Z}$ whilst the set of edges of $\mathcal{Z}$ having two endvertices is denoted by $E(\mathcal{Z})$. Two semiedges are \emph{joined} if they are both deleted and their endvertices are made adjacent. A \emph{$k$-pole} is a multipole with $k$ semiedges. A perfect matching $M$ of a $k$-pole $\mathcal{Z}$ is a subset of generalised edges of $\mathcal{Z}$ such that every vertex of $\mathcal{Z}$ is incident with exactly one generalised edge of $M$. In what follows, we shall construct papillon graphs by joining together semiedges of a number of multipoles. In this sense, given a perfect matching $M$ of a graph $G$, and a multipole $\mathcal{Z}$ used as a building block to construct $G$, we shall say that $M$ contains a semiedge $e$ of the multipole $\mathcal{Z}$, if $M$ contains the edge in $G$ obtained by joining $e$ to another semiedge in the process of constructing $G$.

The $4$-pole $\mathcal{Z}$ with vertex set $\{z_{1}, z_{2},z_{3},z_{4}\}$,  such that $E(\mathcal{Z})$ induces the $4$-cycle $(z_{1}, z_{2},z_{3},z_{4})$ and with exactly one semiedge incident to each of its vertices is referred to as a \emph{$C_{4}$-pole} (see Figure \ref{figure c4pole}). For each $i\in[4]$, let the semiedge incident to $z_{i}$ be denoted by $f_{i}$. The semiedges $f_{1}$ and $f_{2}$ are referred to as the \emph{upper left semiedge} and the \emph{upper right semiedge} of $\mathcal{Z}$, respectively. On the other hand, the semiedges $f_{3}$ and $f_{4}$ are referred to as the \emph{lower left semiedge} and the \emph{lower right semiedge} of $\mathcal{Z}$, respectively (see Figure \ref{figure c4pole}).

\begin{figure}[h]
\centering
\includegraphics[width=0.8\textwidth]{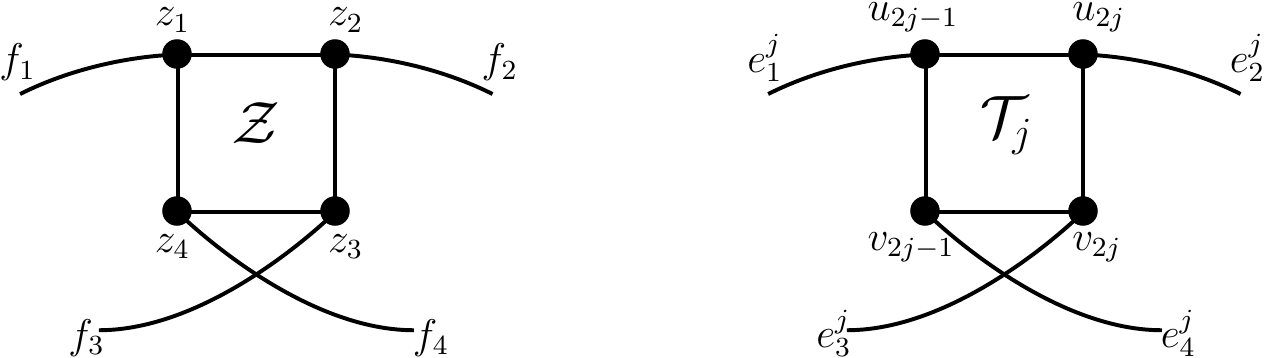}
\caption{A $C_{4}$-pole $\mathcal{Z}$ and the $4$-pole $\mathcal{T}_{j}$ in $\mathcal{P}_{r,\ell}$}
\label{figure c4pole}
\end{figure}

For some integer $n\geq 1$, let $\mathcal{Z}_{1},\ldots, \mathcal{Z}_{n}$ be $n$ copies of the above $C_4$-pole $\mathcal{Z}$. For each $j\in[n]$, let $V(\mathcal{Z}_{j})=\{z_{1}^{j}, z_{2}^{j},z_{3}^{j},z_{4}^{j}\}$, and let $f_1^j,f_2^j,f_3^j,f_4^j$ be the semiedges of $\mathcal{Z}_{j}$ respectively incident to $z_{1}^{j}, z_{2}^{j},z_{3}^{j},z_{4}^{j}$ such that $f_{1}^{j}$ and $f_{2}^{j}$ are the upper left and upper right semiedges of $\mathcal{Z}_{j}$, whilst $f_{3}^{j}$ and $f_{4}^{j}$ are the lower left and lower right semiedges of $\mathcal{Z}_{j}$. A \emph{chain of $C_{4}$-poles} of length $n\geq 2$, is the $4$-pole obtained by respectively joining $f_{2}^{j}$ and $f_{4}^{j}$ (upper and lower right semiedges of $\mathcal{Z}_{j}$) to $f_{1}^{j+1}$ and $f_{3}^{j+1}$ (upper and lower left semiedges of $\mathcal{Z}_{j+1}$), for every $j\in[n-1]$. When $n=1$, a chain of $C_{4}$-poles of length $1$ is just a $C_{4}$-pole. For simplicity, we shall refer to a chain of $C_{4}$-poles of length $n$, as a \emph{$n$-chain of $C_{4}$-poles}, or simply a \emph{$n$-chain}. The semiedges $f_{1}^{1}$ and $f_{3}^{1}$ (similarly, $f_{2}^{n}$ and $f_{4}^{n}$) are referred to as the upper left and lower left (respectively, upper right and lower right) semiedges of the $n$-chain.
A chain of $C_{4}$-poles of any length has exactly four semiedges. For simplicity, when we say that $e_{1},e_{2},e_{3},e_{4}$ are the four semiedges of a chain $\mathcal{Z}'$ of $C_{4}$-poles (possibly of length $1$), we mean that $e_{1}$ and $e_{2}$ are respectively the upper left and upper right semiedges of $\mathcal{Z}'$, whilst $e_{3}$ and $e_{4}$ are respectively the lower left and lower right semiedges of the same chain $\mathcal{Z}'$ (see Figure \ref{figure chain}). The semiedges $e_{1}$ and $e_{2}$ (similarly, $e_{3}$ and $e_{4}$) are referred to collectively as the \emph{upper semiedges} (respectively, \emph{lower semiedges}) of $\mathcal{Z}'$. In a similar way, the semiedges $e_{1}$ and $e_{3}$ (similarly, $e_{2}$ and $e_{4}$) are referred to collectively as the \emph{left semiedges} (respectively, \emph{right semiedges}) of $\mathcal{Z}'$.

\begin{figure}[h]
\centering
\includegraphics[width=.7\textwidth]{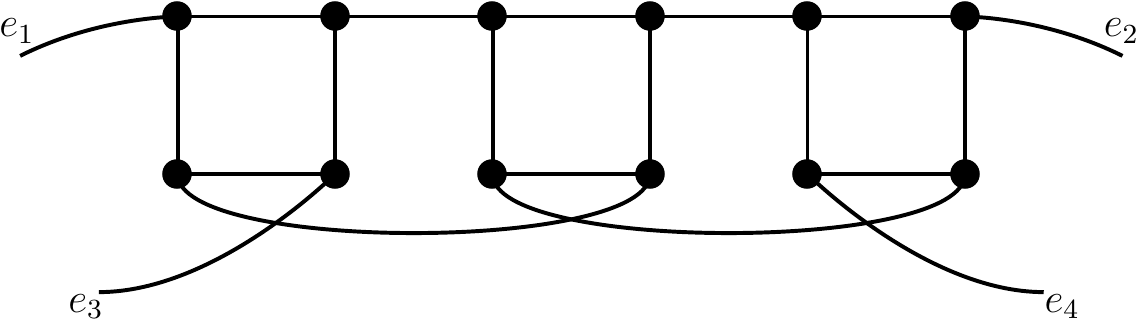}
\caption{A chain of $C_{4}$-poles of length 3 having semiedges $e_{1},e_{2},e_{3},e_{4}$}
\label{figure chain}
\end{figure}

In order to construct the papillon graph $\mathcal{P}_{r,\ell}$ using $C_4$-poles as building blocks, for each $j\in[r+\ell]$, we consider the $4$-pole $\mathcal{T}_j$ arising from the cycle $(u_{2j-1},u_{2j},v_{2j},v_{2j-1})$ of $\mathcal{P}_{r,\ell}$, whose semiedges are $e_1^j, e_2^j, e_3^j, e_4^j$ as in Figure \ref{figure c4pole}. The $r$-chain and $\ell$-chain giving rise to $\mathcal{P}_{r,\ell}$ consist of $\mathcal{T}_{1}, \ldots, \mathcal{T}_{r}$ (referred to as the \emph{right $r$-chain} of $\mathcal{P}_{r,\ell}$), and $\mathcal{T}_{r+1}, \ldots, \mathcal{T}_{r+\ell}$ (referred to as the \emph{left $\ell$-chain} of $\mathcal{P}_{r,\ell}$), which have semiedges $e_1^1, e_2^r, e_3^1, e_4^r$, and $e_1^{r+1}, e_2^{r+\ell},\linebreak e_3^{r+1}, e_4^{r+\ell}$, respectively. The papillon graph $\mathcal{P}_{r,\ell}$ is then obtained by joining the semiedges in pairs as follows: $e_1^1$ to $e_2^{r+\ell}$, $e_2^r$ to $e_1^{r+1}$, $e_3^1$ to $e_3^{r+1}$, and $e_4^r$ to $e_4^{r+\ell}$.

\subsection{Odd cycles and isomorphisms in the class of papillon graphs}\label{section odd}

In this section we shall discuss the presence and behaviour of odd cycles in papillon graphs. Consider the balanced papillon graph $\mathcal{P}_{n}$ and let $C$ be an odd cycle in $\mathcal{P}_{n}$. Since cycles intersect $C_{4}$-poles in $2$, $3$ or $4$ vertices, there must exist some $t_{1}\in[2n]$, such that $|V(C)\cap V(\mathcal{T}_{t_{1}})|=3$. Without loss of generality, assume that $t_{1}\in [n]$, that is, $\mathcal{T}_{t_{1}}$ belongs to the right $n$-chain of $\mathcal{P}_{n}$. If $t_{1}\not\in\{1,n\}$, we must have exactly one of the following:
\begin{itemize}
\item $|V(C)\cap V(\mathcal{T}_{i})|=4$, for all $i\in\{1,\ldots, t_{1}-1\}$; or
\item $|V(C)\cap V(\mathcal{T}_{i})|=4$, for all $i\in\{t_{1}+1,\ldots, n\}$.
\end{itemize}

Without loss of generality, assume that we either have $t_{1}=1$, or $|V(C)\cap V(\mathcal{T}_{i})|=4$, for all $i\in\{1,\ldots, t_{1}-1\}$. This implies that the number of vertices in $C$ belonging to $\cup_{i=1}^{t_{1}}V(\mathcal{T}_{i})$ is odd and at least 3. Moreover, the edges $a$ and $c$ must belong to $C$. We claim that $b\not\in E(C)$. For, suppose that $b\in E(C)$. Since $\mathcal{X}$ is a 4-edge-cut, $d\in E(C)$ as well. This implies that $n>1$ and there exist:
\begin{itemize}
\item $t_{2}\in \{t_{1}+1,\ldots, n\}$, such that $|V(C)\cap V(\mathcal{T}_{t_{2}})|=3$; 
\item $s_{1}\in \{n+1,\ldots, 2n-1\}$, such that $|V(C)\cap V(\mathcal{T}_{s_{1}})|=3$; and
\item $s_{2}\in \{s_{1}+1,\ldots, 2n\}$, such that $|V(C)\cap V(\mathcal{T}_{s_{2}})|=3$.
\end{itemize}
Let $\Omega=\{1,\ldots, t_{1}\}\cup \{t_{2},\ldots, n, n+1,\ldots, s_{1}\}\cup\{s_{2},\ldots, 2n\}$. If $\Omega\setminus\{t_{1},t_{2},s_{1},s_{2}\}\neq \emptyset$, then for any $j\in\Omega\setminus\{t_{1},t_{2},s_{1},s_{2}\}$, $|V(C)\cap V(\mathcal{T}_{j})|=4$. Additionally, for any $k\in[2n]\setminus\Omega$, $|V(C)\cap V(\mathcal{T}_{k})|=0$. However, this means that $C$ has even length, a contradiction. Thus, $\{b,d\}\cap E(C)=\emptyset$. As a result, $C$ intersects none of the $C_{4}$-poles $\mathcal{T}_{t_{1}+1},\ldots, \mathcal{T}_{n}$, but intersects each of the $C_{4}$-poles $\mathcal{T}_{n+1},\ldots, \mathcal{T}_{n}$ in exactly 2 or 4 vertices. Hence, the length of $C$ is at least $2n+3$. When $n=1$, $(u_{1},u_{2},v_{2},v_{4},u_{4})$ is a $5$-cycle, and when $n>1$, $(u_{1},u_{2},v_{2},v_{2n+2}, u_{2n+2},u_{2n+3},u_{2n+4},\ldots, u_{4n})$ is an odd cycle of length exactly $2n+3$. Therefore, a shortest odd cycle in $\mathcal{P}_{n}$ has length $2n+3$. By using similar arguments, a shortest odd cycle in $\mathcal{P}_{r,\ell}$ has length $2r+3$. 

\begin{remark}
The papillon graph $\mathcal{P}_{r,\ell}$ is not bipartite and has a shortest odd cycle of length $2r+3$.
\end{remark}

Consequently, we can show that any two distinct papillon graphs $\mathcal{P}_{r_{1},\ell_{1}}$ and $\mathcal{P}_{r_{2},\ell_{2}}$ are not isomorphic, where by distinct we mean that $(r_{1},\ell_{1})\neq (r_{2},\ell_{2})$. Suppose not, for contradiction. Since $\mathcal{P}_{r_{1},\ell_{1}}\simeq \mathcal{P}_{r_{2},\ell_{2}}$, we must have $r_{1}+\ell_{1}=r_{2}+\ell_{2}$, and so if $r_1=r_2$, then this implies that $\ell_1=\ell_2$, and conversely. Hence, $r_1\neq r_2$ and $\ell_1\neq \ell_2$. 
Thus, without loss of generality, we can assume that $r_{1}<r_{2}$. However, this means that a shortest odd cycle in $\mathcal{P}_{r_{1},\ell_{1}}$ (of length $2r_{1}+3$), is shorter than a shortest odd cycle in $\mathcal{P}_{r_{2},\ell_{2}}$ (of length $2r_{2}+3$), a contradiction. 

We are now in a position to give our first result.

\begin{theorem}\label{theorem papillon e2f}
Every papillon graph $\mathcal{P}_{r,\ell}$ is E2F.
\end{theorem}

\begin{proof}
Let $\mathcal{P}_{r,\ell}$ be a counterexample to the above statement, and let $M$ be a perfect matching of $\mathcal{P}_{r,\ell}$ whose complementary 2-factor contains an odd cycle $C$. As previously discussed, $C$ must intersect some $\mathcal{T}_{j}$, for some $j\in[r+\ell]$, in exactly 3 (consecutive) vertices. Without loss of generality, assume that these 3 vertices are $u_{2j-1}, u_{2j}, v_{2j}$. This means that both the left semiedges ($e_{1}^{j}$ and $e_{3}^{j}$) of $\mathcal{T}_{j}$ belong to this odd cycle. However, since $C$ is in the complementary 2-factor of $M$, the two edges $u_{2j-1}v_{2j-1}$ and $v_{2j-1}v_{2j}$ (which do not belong to $E(C)$) must both belong to $M$, a contradiction.
\end{proof}

\section{The PMH-property in papillon graphs}

\subsection{The balanced case $r=\ell$}

Let $M$ be a perfect matching of the balanced papillon graph $\mathcal{P}_{n}$. Since $\mathcal{X}=\{a,b,c,d\}$ is a $4$-edge-cut of $\mathcal{P}_{n}$, $|M\cap \mathcal{X}|\equiv 0\pmod{2}$, that is, $|M\cap \mathcal{X}|$ is $0, 2$ or $4$. The following is a useful lemma which shall be used frequently in the results that follow.

\begin{lemma}\label{lemma m cap x}
Let $M$ be a perfect matching of the balanced papillon graph $\mathcal{P}_{n}$ and let $\mathcal{X}$ be its principal 4-edge-cut. If $|M\cap \mathcal{X}|=k$, then $|M\cap \partial\mathcal{T}_{j}|=k$, for each $j\in[2n]$.
\end{lemma}

\begin{proof}
Let $M$ be a perfect matching of $\mathcal{P}_{n}$. We first note that the left semiedges of a $C_{4}$-pole are contained in a perfect matching if and only if the right semiedges of the $C_{4}$-pole are contained in the same perfect matching. The lemma is proved by considering three cases depending on the possible values of $k$, that is, $0,2$ or $4$. When $n=1$, the result clearly follows since $\mathcal{X}$ is made up by joining $\partial\mathcal{T}_{1}$ and $\partial\mathcal{T}_{2}$ accordingly. So assume $n\geq 2$.\\

\noindent\textbf{Case I.} $k=0$.\\
Since $a$ and $c$ do not belong to $M$, the left semiedges of $\mathcal{T}_{1}$ are not contained in $M$, and so $M$ cannot contain its right semiedges. Therefore, $|M\cap \partial\mathcal{T}_{1}|=0$. Consequently, the left semiedges of $\mathcal{T}_{2}$ are not contained in $M$ implying again that $|M\cap \partial\mathcal{T}_{2}|=0$. By repeating the same argument up till the $n^{\textrm{th}}$ $C_{4}$-pole, we have that $|M\cap \partial\mathcal{T}_{j}|=0$, for every $j\in[n]$. By noting that $c$ and $d$ do not belong to $M$ and repeating a similar argument to the 4-poles in the left $n$-chain, we can deduce that $|M\cap \partial\mathcal{T}_{j}|=0$ for every $j\in[2n]$.\\

\noindent\textbf{Case II.} $k=4$.\\
Since $a$ and $c$ belong to $M$, the left semiedges of $\mathcal{T}_{1}$ are contained in $M$, and so $M$ contains its right semiedges as well. Therefore, $|M\cap \partial\mathcal{T}_{1}|=4$. Consequently, the left semiedges of $\mathcal{T}_{2}$ are contained in $M$ implying again that $|M\cap \partial\mathcal{T}_{2}|=4$. As in Case I, by noting that both $c$ and $d$ belong to $M$ and repeating a similar argument to the 4-poles in the left $n$-chain, we can deduce that $|M\cap \partial\mathcal{T}_{j}|=4$ for every $j\in[2n]$.\\

\noindent\textbf{Case III.} $k=2$.\\
We first claim that when $k=2$, $M\cap \mathcal{X}$ must be equal to $\{a,d\}$ or $\{b,c\}$. For, suppose that $M\cap \mathcal{X}=\{a,c\}$, without loss of generality. This means that the right semiedges of $\mathcal{T}_{1}$ are also contained in $M$, implying that $|M\cap \partial\mathcal{T}_{1}|=4$. This implies that the left semiedges of $\mathcal{T}_{2}$ are contained in $M$, which forces $|M\cap \partial\mathcal{T}_{j}|$ to be equal to 4, for every $j\in[2n]$. In particular, $|M\cap \partial\mathcal{T}_{n}|=4$, implying that the edges $b$ and $d$ belong to $M$, a contradiction since $M\cap \mathcal{X}=\{a,c\}$. This proves our claim. Since the natural automorphism $\psi$ of $\mathcal{P}_{n}$, which exchanges the outer- and inner-cycles, exchanges also $\{a,d\}$ with $\{b,c\}$, without loss of generality, we may assume that $M\cap \mathcal{X}=\{a,d\}$. Since $c\not\in M$, $1\leq |M\cap \partial\mathcal{T}_{1}|<4$. But, $\partial\mathcal{T}_{1}$ corresponds to a 4-edge-cut in $\mathcal{P}_{n}$, and so, by using a parity argument, $|M\cap \partial\mathcal{T}_{1}|$ must be equal to 2, implying that exactly one of the right semiedges of $\mathcal{T}_{1}$ is contained in $M$. This means that exactly one left semiedge of $\mathcal{T}_{2}$ is contained in $M$, and consequently, by a similar argument now applied to $\mathcal{T}_{2}$, we obtain $|M\cap \partial\mathcal{T}_{2}|=2$. By repeating the same argument and noting that $\mathcal{T}_{n+1}$ has exactly one left semiedge (corresponding to the edge $d$) contained in $M$, one can deduce that $|M\cap \partial\mathcal{T}_{j}|=2$ for every $j\in[2n]$.
\end{proof}

The following two results are two consequences of the above lemma and they both follow directly from the proof of Case III. In a few words, if a perfect matching $M$ of $\mathcal{P}_n$ intersects its principal 4-edge-cut in exactly two of its edges, then these two edges are either the pair $\{a,d\}$ or the pair $\{b,c\}$, and, for every $j\in [2n]$,  $M$ contains only one pair of semiedges of $\mathcal{T}_j$ which does not consist of the pair of left semiedges of $\mathcal{T}_j$ nor the pair of right semiedges of $\mathcal {T}_j$.
\begin{corollary}\label{cor ad or bc}
Let $M$ be a perfect matching of $\mathcal{P}_{n}$ and let $\mathcal{X}$ be its principal 4-edge-cut. If $|M\cap \mathcal{X}|=2$, then $M\cap \mathcal{X}$ is equal to $\{a,d\}$ or $\{b,c\}$.
\end{corollary}

\begin{corollary}\label{cor 2 pm and c4 poles}
Let $M$ be a perfect matching of $\mathcal{P}_{n}$ and let $\mathcal{X}$ be its principal 4-edge-cut such that $|M\cap \mathcal{X}|=2$. For each $j\in[2n]$, $M$ contains exactly one of the following sets of semiedges: $\{e_1^j,e_2^j\}, \{e_3^j,e_4^j\}, \{e_1^j,e_4^j\}, \{e_2^j,e_3^j\}$, that is, of all possible pairs of semiedges of $\mathcal{T}_{j}$, $\{e_1^j,e_3^j\}$ and $\{e_2^j,e_4^j\}$ cannot be contained in $M$.
\end{corollary}

In the sequel, the process of traversing one path after another shall be called \emph{concatenation of paths}. If two paths $P$ and $Q$ have endvertices $x,y$ and $y,z$, respectively, we write $PQ$ to denote the path starting at $x$ and ending at $z$ obtained by traversing $P$ and then $Q$. Without loss of generality, if $x$ is adjacent to $y$, that is, $P$ is a path on two vertices, we may write $xyQ$ instead of $PQ$.

\begin{lemma}\label{lemma complementary}
Let $M_{1}$ be a perfect matching of $\mathcal{P}_{n}$ such that $|M_{1}\cap \mathcal{X}|=2$.
\begin{enumerate}[(i)]
\item There exists a perfect matching $M_{2}$ of $\mathcal{P}_n$ such that $|M_{2}\cap \mathcal{X}|=2$ and $M_{1}\cap M_{2}=\emptyset$.
\item The complementary $2$-factors of $M_{1}$ and $M_{2}$ are both Hamiltonian cycles.
\end{enumerate}
\end{lemma}
\begin{proof}
(i)\,\, Since $|M_{1}\cap \mathcal{X}|=2$, by Lemma \ref{lemma m cap x} we get that $|M_{1}\cap \partial\mathcal{T}_{j}|=2$ for every $j\in[2n]$. For each $j$, let $P^{(j)}$ be the subgraph of $\mathcal{P}_{n}$ which is induced by $E(\mathcal{T}_{j})-M_{1}$. Note that $\cup_{j=1}^{2n} V(P^{(j)})=V(\mathcal{P}_{n})$. By Corollary \ref{cor 2 pm and c4 poles}, each $P^{(j)}$ is a path of length 3. Letting $N$ be the unique perfect matching of $\mathcal{P}_{n}$ which intersects each $E(P^{(j)})$ in exactly two edges, we note that $M_{1}\cap N=\emptyset$. Let $M_{2}=E(\mathcal{P}_{n})-(M_{1}\cup N)$. Since $M_{1}$ and $N$ are two disjoint perfect matchings, $M_{2}$ is also a perfect matching of $\mathcal{P}_{n}$ and, in particular, $M_{2}$ contains $\mathcal{X}-(M_{1}\cap \mathcal{X})$. Thus, $|M_{2}\cap \mathcal{X}|=2$ and $M_{1}\cap M_{2}=\emptyset$, proving part (i).

\begin{figure}[h]
\centering
\includegraphics[width=0.7\textwidth]{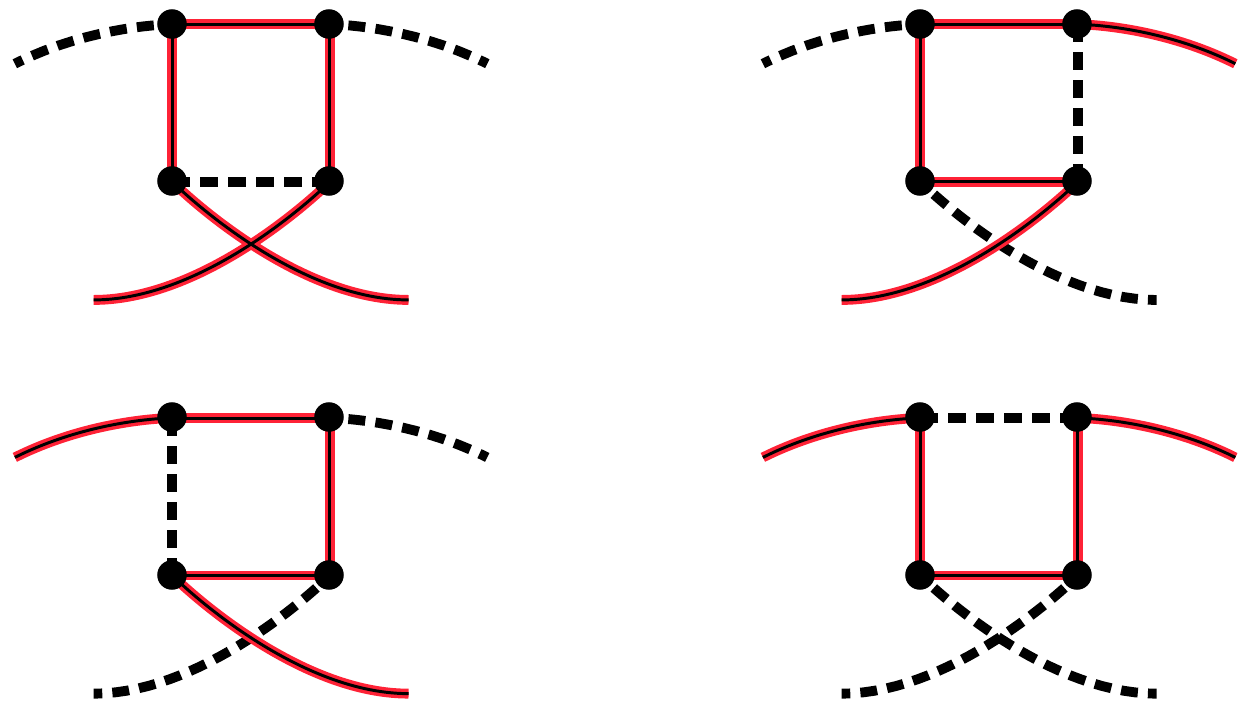}
\caption{Perfect matching $M_{1}$ (bold dashed edges) with $|M_{1}\cap \mathcal{X}|=2$ and its complementary 2-factor (highlighted edges)}
\label{figure comp 2factor lemma}
\end{figure}
(ii)\,\, Let $M_{2}$ be as in part (i), that is, $|M_{2}\cap \mathcal{X}|=2$ and $M_{1}\cap M_{2}=\emptyset$. When $n=1$, the result clearly follows. So assume $n\geq 2$. For distinct $i$ and $j$ in $[2n]$, let $Q^{(i,j)}$ be the subgraph of $\mathcal{P}_{n}$ which is induced by $M_{2}\cap \{xy\in E(\mathcal{P}_{n}): x\in V(\mathcal{T}_{i}), y\in V(\mathcal{T}_{j})\}$, that is, $E(Q^{(i,j)})$ is either empty or consists of exactly one edge, that is, $Q^{(i,j)}$ is a path of length 1. When $M_{1}\cap \mathcal{X}=\{a,d\}$, we can form a Hamiltonian cycle of $\mathcal{P}_{n}$ (not containing $M_{1}$) by considering the following concatenation of paths:
\begin{linenomath}
$$P^{(1)}Q^{(1,2)}\ldots Q^{(n-1,n)}P^{(n)}Q^{(n,2n)}P^{(2n)}Q^{(2n,2n-1)}\ldots P^{(n+1)}Q^{(n+1,1)},$$
\end{linenomath}
where $Q^{(1,2)}$ and $Q^{(2n,2n-1)}$ are respectively followed by $P^{(2)}$ and $P^{(2n-1)}$, and, $Q^{(n,2n)}$ and $Q^{(n+1,1)}$ consist of the edges $b$ and $c$, respectively. On the other hand, when $M_{1}\cap \mathcal{X}=\{b,c\}$, we can form a Hamiltonian cycle of $\mathcal{P}_{n}$ (not containing $M_{1}$) by considering the following concatenation of paths:
\begin{linenomath}
$$P^{(1)}Q^{(1,2)}\ldots Q^{(n-1,n)}P^{(n)}Q^{(n,n+1)}P^{(n+1)}Q^{(n+1,n+2)}\ldots P^{(2n)}Q^{(2n,1)},$$
\end{linenomath}
where $Q^{(1,2)}$ and $Q^{(n+1,n+2)}$ are respectively followed by $P^{(2)}$ and $P^{(n+2)}$, and, $Q^{(n,n+1)}$ and $Q^{(2n,1)}$ consist of the edges $d$ and $a$, respectively. Thus, the complementary 2-factor of $M_{1}$ is a Hamiltonian cycle. This is depicted in Figure \ref{figure comp 2factor lemma}. The proof that the complementary 2-factor of $M_{2}$ is a Hamiltonian cycle follows analogously.
\end{proof}

\begin{proposition}\label{prop papillon not pmh}
Let $n$ be a positive odd integer. Then, the balanced papillon graph $\mathcal{P}_{n}$ is not PMH.
\end{proposition}

\begin{proof}
Consider the following perfect matching of the balanced papillon graph $\mathcal{P}_{n}$:
\[M=\cup_{i=1}^{2n}\{u_{2i-1}u_{2i}, v_{2i-1}v_{2i}\}.\] It is clear that when $n=1$, the perfect matching $M$ cannot be extended to a Hamiltonian cycle of the balanced papillon graph $\mathcal{P}_{1}$. So assume that $n\geq 3$. We claim that $M$ cannot be extended to a Hamiltonian cycle of $\mathcal{P}_{n}$.
For, let $F$ be a $2$-factor of $\mathcal{P}_{n}$ containing $M$. Since $u_1u_2 \in M$ and $\mathcal{P}_{n}$ is cubic, $F$ contains exactly one of the following two edges: $u_{1}u_{4n}$ or $u_{1}v_{1}$.
In the former case, if $u_{1}u_{4n} \in E(F)$, then, $u_{2n}u_{2n+1}$ and  all the edges of the outer- and inner-cycle will belong to $F$ (at the same time, the choice of $u_{1}u_{4n}$ forbids all the spokes of $\mathcal{P}_{n}$ to belong to $F$), yielding two disjoint cycles each of length $4n$. In the latter case, if $u_{1}v_{1}\in E(F)$, then $F$ must also contain all spokes $u_{i}v_{i}$, for $1<i\leq 4n$. In fact, the subgraph induced by the set of spokes is exactly the complement of the 2-factor obtained in the former case. Consequently, $F$ will consist of $2n$ disjoint $4$-cycles. 
\end{proof}

Consider $\mathcal{P}_{n}$, with $n\geq 2$, and let $M$ be a perfect matching of $\mathcal{P}_{n}$ with $M \cap \mathcal{X}=0$, which by Lemma \ref{lemma m cap x} implies that $|M\cap \partial\mathcal{T}_{j}|=0$ for all $j\in[2n]$. Now consider $j\in [2n]\setminus\{n,2n\}$ and let $\mathcal{T}_{(j,j+1)}$ denote a $2$-chain composed of $\mathcal{T}_{j}$ and $\mathcal{T}_{j+1}$.
We say that $\mathcal{T}_{(j,j+1)}$ is \emph{symmetric with respect to $M$} if exactly one of the following occurs:
\begin{enumerate}[(i)]
\item $\{u_{2j-1}v_{2j-1},u_{2j}v_{2j},u_{2j+1}v_{2j+1},u_{2j+2}v_{2j+2}\}\subset M$; or
\item $\{u_{2j-1}u_{2j},v_{2j-1}v_{2j},u_{2j+1}u_{2j+2},v_{2j+1}v_{2j+2}\}\subset M$.
\end{enumerate}
If neither (i) nor (ii) occur, $\mathcal{T}_{(j,j+1)}$ is said to be \emph{asymmetric with respect to $M$}. This is shown in Figure \ref{figure sym asym definition}.

\begin{figure}[h]
\centering
\includegraphics[width=1\textwidth]{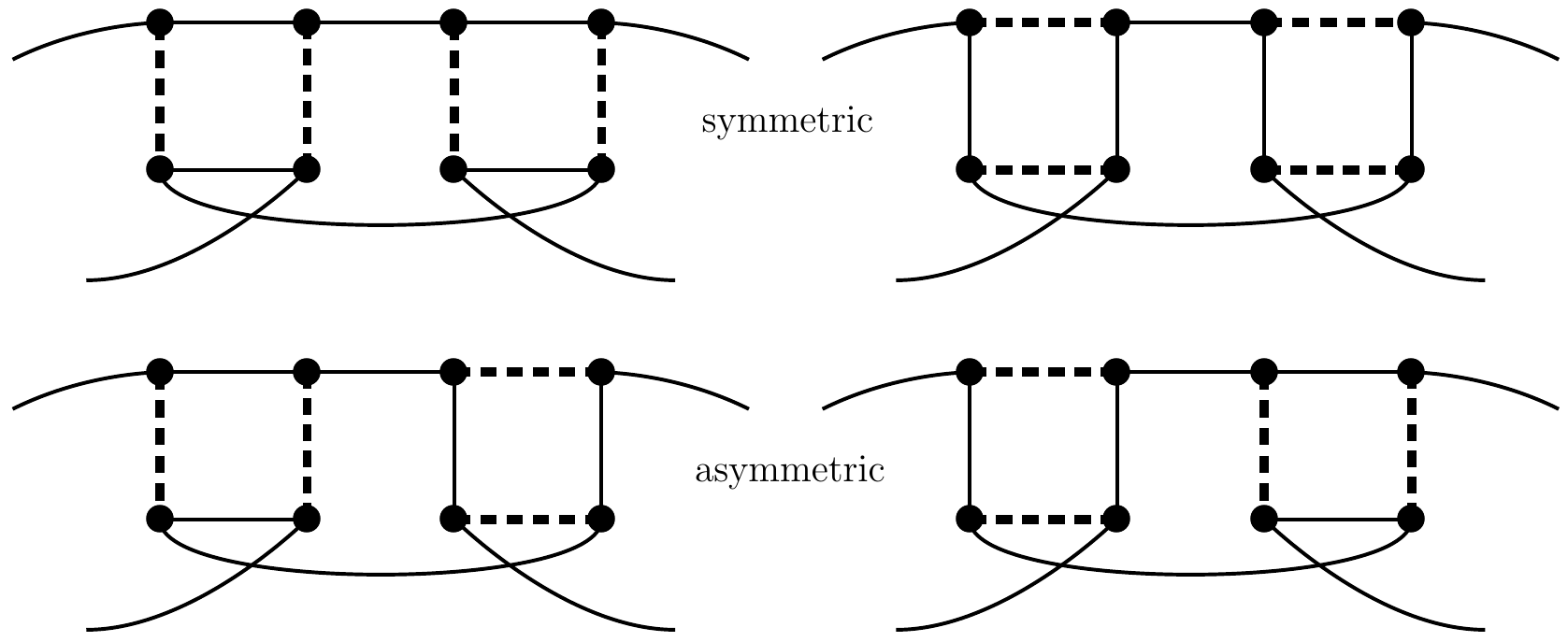}
\caption{Symmetric and asymmetric $2$-chains with the bold dashed edges belonging to $M$}
\label{figure sym asym definition}
\end{figure}

\begin{remark}\label{remark m cap x 2}
Let $n\geq 2$. Consider a perfect matching $M_{1}$ of $\mathcal{P}_{n}$ such that $M_{1}$ does not intersect the principal 4-edge-cut $\mathcal{X}$ of $\mathcal{P}_{n}$, that is, $M_{1}\cap \mathcal{X}=\emptyset$, and consider a $2$-chain of $\mathcal{P}_{n}$, say $\mathcal{T}_{(j,j+1)}$ with $j\in [2n]\setminus\{n,2n\}$, having semiedges $e_1,e_2,e_3,e_4$, where $e_1=e_{1}^{j},e_2=e_{2}^{j+1},e_3=e_{3}^{j}$ and $e_4=e_{4}^{j+1}$.
Assume there exists a perfect matching  $M_{2}$ of $\mathcal{P}_{n}$ such that $|M_{2}\cap \mathcal{X}|=2$ and $M_{1}\cap M_{2}=\emptyset$ (see Figure \ref{figure remark1}). If $\mathcal{T}_{(j,j+1)}$ is symmetric with respect to $M_{1}$, then we have exactly one of the following instances:
\begin{linenomath}
$$M_{2} \cap \partial\mathcal{T}_{(j,j+1)}=\{e_1,e_2\} \mbox{ (upper); \quad or \quad} M_{2} \cap \partial\mathcal{T}_{(j,j+1)}=\{e_3,e_4\} \mbox{ (lower)}.$$\end{linenomath} 
Otherwise, if $\mathcal{T}_{(j,j+1)}$ is asymmetric with respect to $M_{1}$, then exactly one of the following must occur:
\[M_{2} \cap \partial\mathcal{T}_{(j,j+1)}=\{e_1,e_4\} \mbox{ (upper left, lower right); or }\]
\[M_{2} \cap \partial\mathcal{T}_{(j,j+1)}=\{e_2,e_3\} \mbox{ (upper right, lower left)}.\]

Notwithstanding whether $\mathcal{T}_{(j,j+1)}$ is symmetric or asymmetric with respect to $M_{1}$, $(M_{1}\cup M_{2})\cap E(\mathcal{T}_{(j,j+1)})$ induces a path (see Figure \ref{figure remark1}) which contains all the vertices of $V(\mathcal{T}_{(j,j+1)})$, and whose endvertices are the endvertices of the semiedges in $M_{2}\cap\partial\mathcal{T}_{(j,j+1)}$.

\begin{figure}[h]
\centering
\includegraphics[width=1\textwidth]{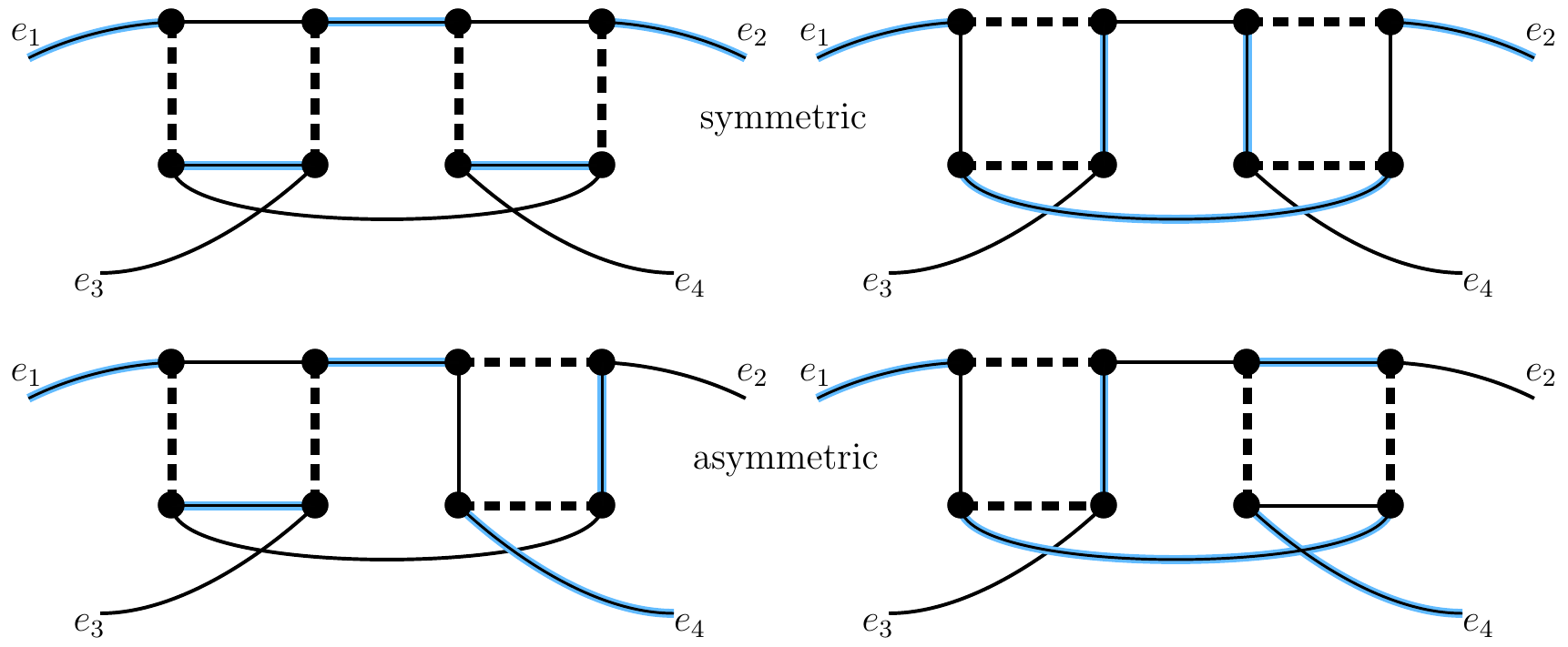}
\caption{$2$-chains when $M_{1}\cap \mathcal{X}=\emptyset$ and $|M_{2}\cap \mathcal{X}|=2$ (bold dashed edges belong to $M_{1}$ and highlighted edges to $M_{2}$)}
\label{figure remark1}
\end{figure}
\end{remark}

\begin{remark}\label{remark m cap x 4}
Let $n\geq 2$. Consider a perfect matching $M_{1}$ of $\mathcal{P}_{n}$ such that $M_{1}$ does not intersect the principal 4-edge-cut $\mathcal{X}$ of $\mathcal{P}_{n}$, that is, $M_{1}\cap \mathcal{X}=\emptyset$, and consider a $2$-chain of $\mathcal{P}_{n}$, say $\mathcal{T}_{(j,j+1)}$ with $j\in [2n]\setminus\{n,2n\}$.
Let $M_{2}$ be the perfect matching of $\mathcal{P}_{n}$ such that $|M_{2}\cap \mathcal{X}|=4$. Clearly $M_{1}\cap M_{2}=\emptyset$. Notwithstanding whether $\mathcal{T}_{(j,j+1)}$ is symmetric or asymmetric with respect to $M_{1}$, we have that $(M_{1}\cup M_{2})\cap E(\mathcal{T}_{(j,j+1)})$ induces two disjoint paths of equal length (see Figure \ref{figure remark2}) whose union contains all the vertices of $\mathcal{T}_{j}$ and $\mathcal{T}_{j+1}$. Let $Q$ be one of these paths. We first note that $Q$ contains exactly one vertex from $\{u_{j},v_{j+1}\}$ and exactly one vertex from $\{u_{j+3},v_{j+2}\}$. If $\mathcal{T}_{(j,j+1)}$ is symmetric with respect to $M_{1}$, then $Q$ contains $u_{j}$ if and only if $Q$ contains $u_{j+3}$. Otherwise, if $\mathcal{T}_{(j,j+1)}$ is asymmetric with respect to $M_{1}$, then $Q$ contains $u_{j}$ if and only if $Q$ contains $v_{j+2}$.

\begin{figure}[h!]
\centering
\includegraphics[width=1\textwidth]{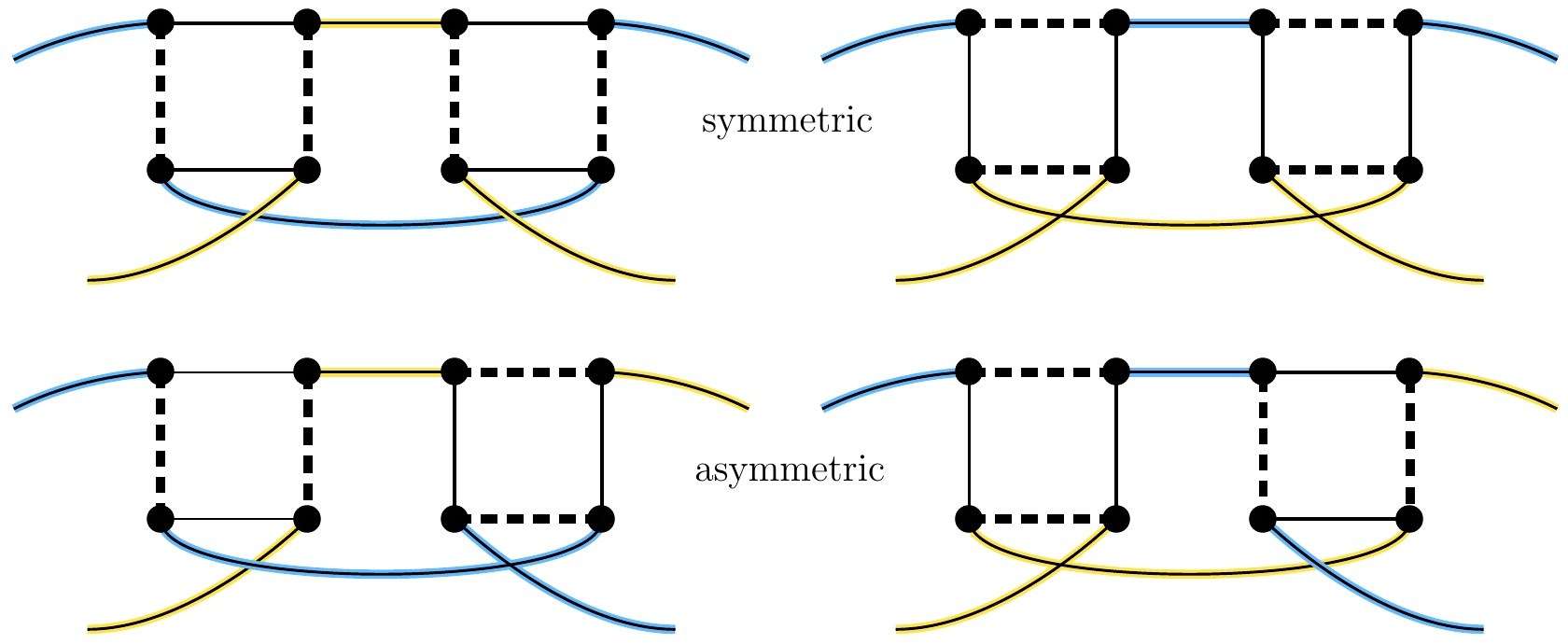}
\caption{$2$-chains when $M_{1}\cap \mathcal{X}=\emptyset$ and $|M_{2}\cap \mathcal{X}|=4$ (bold dashed edges belong to $M_{1}$ and highlighted edges to $M_{2}$)}
\label{figure remark2}
\end{figure}

\end{remark}

\begin{theorem}\label{theorem papillon pmh}
Let $n$ be a positive even integer. Then, the balanced papillon graph $\mathcal{P}_{n}$ is PMH.
\end{theorem}

\begin{proof}
Let $M_{1}$ be a perfect matching of $\mathcal{P}_{n}$. We need to show that there exists a perfect matching $M_{2}$ of $\mathcal{P}_{n}$ such that $M_{1}\cup M_{2}$ induces a Hamiltonian cycle of $\mathcal{P}_{n}$. Three cases, depending on the intersection of $M_{1}$ with the principal 4-edge-cut $\mathcal{X}$ of $\mathcal{P}_{n}$, are considered. If $|M_{1}\cap \mathcal{X}|=2$, then, by Lemma \ref{lemma complementary}, there exists a perfect matching $N$ of $\mathcal{P}_{n}$ such that $|N\cap \mathcal{X}|=2$ and $M_{1}\cap N=\emptyset$. Moreover, the complementary 2-factor of $N$ is a Hamiltonian cycle. Since $M_{1}$ is contained in the mentioned 2-factor, the result follows. When $|M_{1}\cap \mathcal{X}|=4$, we can define $M_{2}$ to be the following perfect matching:
\begin{linenomath}
$$M_{2}=\{u_{1}v_{1},u_{2}v_{2}\}\bigcup \cup_{j=2}^{2n}\{u_{2j-1}u_{2j},v_{2j-1}v_{2j}\}.$$
\end{linenomath}
In fact, $M_{1}\cup M_{2}$ induces the following Hamiltonian cycle: $(u_{1},v_{1},v_{4},\ldots, v_{2n},v_{2n-1}, \linebreak v_{4n-1}, v_{4n},v_{4n-3},\ldots, v_{2n+1},v_{2n+2},v_{2}, u_{2}, u_{3},u_{4},\ldots, u_{4n})$, where $v_{4}$ and $v_{4n-3}$ are respectively followed by $v_{3}$ and $v_{4n-2}$.

What remains to be considered is the case when $|M_{1}\cap \mathcal{X}|=0$. 
Clearly, $|M_{2}\cap \mathcal{X}|$ cannot be zero, because, if so, choosing $M_{2}$ to be disjoint from $M_{1}$, $M_{1}\cup M_{2}$ induces $2n$ disjoint $4$-cycles. Therefore, $|M_{2}\cap \mathcal{X}|$  must be equal to 2 or 4. Let $\mathcal{R}=\{\mathcal{T}_{(1,2)},\ldots, \mathcal{T}_{(n-1,n)}\}$ and  $\mathcal{L}=\{\mathcal{T}_{(n+1,n+2)},\ldots, \mathcal{T}_{(2n-1,2n)}\}$ be the sets of 2-chains within the left and right $n$-chains of $\mathcal{P}_{n}$---namely the right and left $n$-chains each split into $\frac{n}{2}$ $2$-chains. We consider two cases depending on the parity of the number of $2$-chains in $\mathcal{L}$ and $\mathcal{R}$ which are asymmetric with respect to $M_{1}$. Let the function $\Phi:\mathcal{\mathcal{R}}\cup \mathcal{L}\rightarrow\{-1,+1\}$ be defined on the $2$-chains $\mathcal{T}\in\mathcal{R}\cup \mathcal{L}$ such that:
\[
 \Phi(\mathcal{T}) =
  \begin{cases}
  +1 & $if $\mathcal{T}$ is symmetric with respect to $M_{1},\\
  -1 & $otherwise.$
  \end{cases}
\]

\noindent\textbf{Case 1.} $\mathcal{L}$ and $\mathcal{R}$ each have an even number (possibly zero) of asymmetric $2$-chains with respect to $M_{1}$.

We claim that there exists a perfect matching such that its union with $M_{1}$ gives a Hamiltonian cycle of $\mathcal{P}_{n}$. Since the number of asymmetric $2$-chains in $\mathcal{R}$ is even, $\prod_{\mathcal{T}\in\mathcal{R}}\Phi(\mathcal{T})=+1$, and consequently, by appropriately concatenating paths as in Remark \ref{remark m cap x 2}, there exists a path $R$ with endvertices $u_{1}$ and $u_{2n}$ whose vertex set is $\cup_{i=1}^{2n}\{u_{i},v_{i}\}$ such that it contains all the edges in $M_{1}\cap (\cup_{i=1}^{n}E(\mathcal{T}_{i}) )$. We remark that this path intersects exactly one edge of $\{xy\in E(\mathcal{P}_{n}): x\in V(\mathcal{T}_{j}), y\in V(\mathcal{T}_{j+1})\}$, for each $j\in [n-1]$.  By a similar reasoning, since $\prod_{\mathcal{T}\in\mathcal{L}}\Phi(\mathcal{T})=+1$, there exists a path $L$ with endvertices $u_{2n+1}$ and $u_{4n}$ whose vertex set is $\cup_{i=2n+1}^{4n}\{u_{i},v_{i}\}$, such that it contains all the edges in $M_{1}\cap (\cup_{i=n+1}^{2n}E(\mathcal{T}_{i}) )$. Once again, this path intersects exactly one edge of $\{xy\in E(\mathcal{P}_{n}): x\in V(\mathcal{T}_{j}), y\in V(\mathcal{T}_{j+1})\}$, for each $j \in\{n+1,\ldots,2n-1\}$.  These two paths, together with the edges $a$ and $d$ form the required Hamiltonian cycle of $\mathcal{P}_{n}$ containing $M_{1}$, proving our claim. We remark that this shows that there exists a perfect matching $M_{2}$ of $\mathcal{P}_{n}$ such that $M_{2}\cap \mathcal{X}=\{a,d\}$, $M_{1}\cap M_{2}=\emptyset$ and with $M_{1}\cup M_{2}$ inducing a Hamiltonian cycle of $\mathcal{P}_{n}$. One can similarly show that there exists a perfect matching $M_{2}'$ of $\mathcal{P}_n$ such that $M_{2}'\cap \mathcal{X}=\{b,c\}$, $M_{1}\cap M_{2}'=\emptyset$ and with $M_{1}\cup M_{2}'$ inducing a Hamiltonian cycle of $\mathcal{P}_{n}$.\\

\noindent\textbf{Case 2.} One of $\mathcal{L}$ and $\mathcal{R}$ has an odd number of asymmetric $2$-chains with respect to $M_{1}$.

Without loss of generality, assume that $\mathcal{R}$ has an odd number of asymmetric $2$-chains with respect to $M_{1}$, that is, $\prod_{\mathcal{T}\in\mathcal{R}}\Phi(\mathcal{T})=-1$. Let $M_{2}$ be the perfect matching of $\mathcal{P}_{n}$ such that $|M_{2}\cap \mathcal{X}|=4$. We claim that $M_{1}\cup M_{2}$ induces a Hamiltonian cycle of $\mathcal{P}_{n}$. Since $\prod_{\mathcal{T}\in\mathcal{R}}\Phi(\mathcal{T})=-1$, by appropriately concatenating paths as in Remark \ref{remark m cap x 4} we can deduce that $M_{1}\cup M_{2}$ contains the edge set of two disjoint paths $R_{1}$ and $R_{2}$, such that:
\begin{enumerate}[(i)]
\item $|V(R_{1})|=|V(R_{2})|=2n$;
\item $V(R_{1})\cup V(R_{2})=\cup_{i=1}^{2n}\{u_{i},v_{i}\}$;
\item the endvertices of $R_{1}$ are $u_{1}$ and $v_{2n-1}$; and
\item the endvertices of $R_{2}$ are $v_{2}$ and $u_{2n}$.
\end{enumerate}
Next, we consider two subcases depending on the value of $\prod_{\mathcal{T}\in\mathcal{L}}\Phi(\mathcal{T})$. We shall be using the fact that
$\{u_{1}u_{4n}, v_{2n-1}v_{4n-1}, v_{2}v_{2n+2}, u_{2n}u_{2n+1} \}=\{a,b,c,d\} = \mathcal{X} \subset M_2$.\\

\noindent\textbf{Case 2a)} $\prod_{\mathcal{T}\in\mathcal{L}}\Phi(\mathcal{T})=-1$

As above, by Remark \ref{remark m cap x 4}, we can deduce that $M_{1}\cup M_{2}$ contains the edge set of two disjoint paths $L_{1}$ and $L_{2}$, such that:
\begin{enumerate}[(i)]
\item $|V(L_{1})|=|V(L_{2})|=2n$;
\item $V(L_{1})\cup V(L_{2})=\cup_{i=2n+1}^{4n}\{u_{i},v_{i}\}$;
\item the endvertices of $L_{1}$ are $u_{2n+1}$ and $v_{4n-1}$; and
\item the endvertices of $L_{2}$ are $v_{2n+2}$ and $u_{4n}$.
\end{enumerate}
The concatenation of the following paths and edges gives a Hamiltonian cycle of $\mathcal{P}_{n}$ containing $M_{1}$:
\begin{linenomath}
$$R_{1} v_{2n-1}v_{4n-1} L_{1} u_{2n+1}u_{2n} R_{2} v_{2}v_{2n+2} L_{2} u_{4n}u_{1}.$$
\end{linenomath}

\noindent\textbf{Case 2b)} $\prod_{\mathcal{T}\in\mathcal{L}}\Phi(\mathcal{T})=+1$.

Once again, by Remark \ref{remark m cap x 4} we can deduce that $M_{1}\cup M_{2}$ contains the edge set of  two disjoint paths $L_{1}$ and $L_{2}$, such that:
\begin{enumerate}[(i)]
\item $|V(L_{1})|=|V(L_{2})|=2n$;
\item $V(L_{1})\cup V(L_{2})=\cup_{i=2n+1}^{4n}\{u_{i},v_{i}\}$;
\item the endvertices of $L_{1}$ are $u_{2n+1}$ and $u_{4n}$; and
\item the endvertices of $L_{2}$ are $v_{2n+2}$ and $v_{4n-1}$.
\end{enumerate}
The concatenation of the following paths and edges gives a Hamiltonian cycle of $\mathcal{P}_{n}$ containing $M_{1}$:
\begin{linenomath}
$$R_{1} v_{2n-1}v_{4n-1} L_{2} v_{2n+2}v_{2} R_{2} u_{2n}u_{2n+1} L_{1} u_{4n}u_{1}.$$
\end{linenomath}
This completes the proof.
\end{proof}

\subsection{The unbalanced case $r<\ell$ and final remarks}\label{section unbalanced}

By following the proofs in Section \ref{section papillon graphs}, the results obtained for balanced papillon graphs are now extended to unbalanced papillon graphs.

\begin{theorem}\label{pmh unbalanced}
The unbalanced papillon graph $\mathcal{P}_{r,\ell}$ is \textit{PMH} if and only if  $r$ and $\ell$ are both even.
\end{theorem}
\begin{proof}
This is an immediate consequence of Proposition \ref{prop papillon not pmh} and Theorem \ref{theorem papillon pmh}. In particular, when at least one of $r$ and $\ell$ is odd, $\mathcal{P}_{r,\ell}$ is not PMH because the perfect matching $\cup_{i=1}^{r+\ell}\{u_{2i-1}u_{2i},v_{2i-1}v_{2i}\}$ of $\mathcal{P}_{r,\ell}$ (illustrated in Figure \ref{figure unbalanced}) cannot be extended to a Hamiltonian cycle.
\end{proof}

\begin{corollary}
The papillon graph $\mathcal{P}_{r,\ell}$ is PMH if and only if $r$ and $\ell$ are both even.
\end{corollary}

\begin{figure}[h]
\centering
\includegraphics[width=1\textwidth]{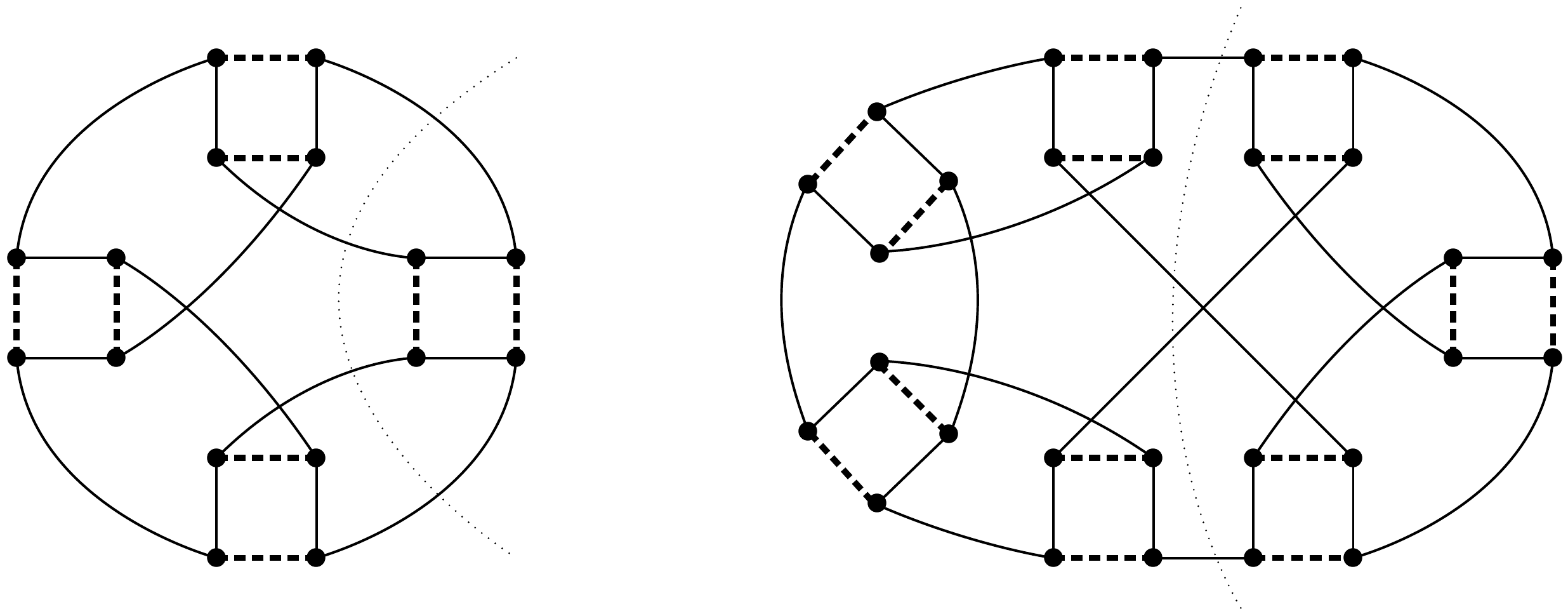}
\caption{$\mathcal{P}_{1,3}$ and $\mathcal{P}_{3,4}$: unbalanced papillon graphs are not always PMH. The above perfect matchings do not extend to a Hamiltonian cycle.}
\label{figure unbalanced}
\end{figure}

Finally, we remark that since $\mathcal{P}_{n}$ is PMH for every even $n\in\mathbb{N}$, balanced papillon graphs provide us with examples of non-bipartite PMH cubic graphs which are cyclically 4-edge-connected and have girth 4 such that their order is a multiple of $16$. Additionally, by considering unbalanced papillon graphs, say $\mathcal{P}_{2,\ell}$, for some even $\ell>2$, we can obtain non-bipartite PMH cubic graphs having the above characteristics (that is, cyclically 4-edge-connected and having girth $4$) such that their order is $8\nu$, for odd $\nu\geq 3$. 

It would also be very compelling to see whether there exist other $4$-poles instead of the $C_{4}$-poles that can be used as building blocks when constructing papillon graphs and which yield non-bipartite PMH or just E2F cubic graphs.

\end{document}